\documentclass[twoside]{article}

\usepackage{style}
\usepackage{amssymb}
\newcommand*{\qed}{\hfill\ensuremath{\square}}%

\usepackage[round]{natbib}

\usepackage{article_packages}
\usepackage[ruled, linesnumbered]{algorithm2e}
\usepackage{siunitx}
\usepackage{balance}
\usepackage{siunitx}
\usepackage{longtable}
\usepackage{lscape}
\usepackage{csvsimple}
\usepackage{caption}

\begin{document}

%

%

\twocolumn[

\aistatstitle{Optimal Approximation of Doubly Stochastic Matrices}

\aistatsauthor{ Nikitas Rontsis \And Paul J. Goulart }

\aistatsaddress{ Department of Engineering Science \\
University of Oxford, UK \\
nikitas.rontsis@eng.ox.ac.uk
\And Department of Engineering Science \\
University of Oxford, UK\\
paul.goulart@eng.ox.ac.uk}]

\begin{abstract}
  We consider the least-squares approximation of a matrix $C$ in the set of doubly stochastic matrices with the same sparsity pattern as $C$. Our approach is based on applying the well-known Alternating Direction Method of Multipliers (ADMM) to a reformulation of the original problem. Our resulting algorithm requires an initial Cholesky factorization of a positive definite matrix that has the same sparsity pattern as $C + I$ followed by simple iterations whose complexity is linear in the number of nonzeros in C, thus ensuring excellent scalability and speed. We demonstrate the advantages of our approach in a series of experiments on problems with up to 82 million nonzeros; these include normalizing large scale matrices arising from the 3D structure of the human genome, clustering applications, and the SuiteSparse matrix library. Overall, our experiments illustrate the outstanding scalability of our algorithm; matrices with millions of nonzeros can be approximated in a few seconds on modest desktop computing hardware.
\end{abstract}

\section{INTRODUCTION}
Consider the following optimization problem
\begin{equation}\label{eqn:approximation_problem} \tag{$\mathcal{P}$}
  \begin{array}{ll}
    \mbox{minimize}   & \frac{1}{2}\norm{X - C}^2 \\
    \mbox{subject to} & X \; \mbox{nonnegative} \\
                      & X_{i,j} = 0 \; \forall \; (i,j) \; \mbox{with} \; C_{i,j} = 0 \\
                      & X \mathbf{1} = \mathbf{1}, \;\;X^T \mathbf{1} = \mathbf{1}, \\
 \end{array}
\end{equation}
which approximates the symmetric real, $n \times n$ matrix $C$
in the {Frobenius norm} by a \emph{doubly stochastic} matrix $X \in \mathbb{R}^{n \times n}$, i.e.,~a matrix with nonnegative elements whose columns and rows sum to one, that has the same sparsity pattern as $C$.
Problem \ref{eqn:approximation_problem} is a \emph{matrix nearness} problem, i.e.,~a problem of finding a matrix with certain properties that is close to some given matrix; see \cite{Higham1989} for a survey on matrix nearness problems. 

Adjusting a matrix so that it becomes doubly stochastic is relevant in many fields, e.g., for preconditioning linear systems \citep{Knight2008}, as a normalization tool used in spectral clustering \citep{Zass2007}, and image filtering \citep{Milanfar2013}, or as a tool to estimate a doubly stochastic matrix from incomplete or approximate data used e.g.,~in longitudinal studies in life sciences \citep{Diggle2002}, or to analyze the 3D structure of the human genome \citep{Rao2014}

A related and widely used approach is to search for a diagonal matrix $D$ such that $DCD$ is doubly stochastic. This is commonly referred to as the \emph{matrix balancing} problem \citep{Knight2008}. Such a scaling matrix $D$ exists and is unique whenever $C$ has total support.
Perhaps surprisingly, when $C$ has only nonnegative elements, then the matrix balancing problem can be considered as a matrix nearness problem. This is because, $DXD$ has been shown to minimize the relative entropy measure \citep[Observation 3.19]{Idel2016}, i.e.,~it is the solution of the following convex problem
\begin{equation}\label{eqn:matrix_balancing}
  \begin{array}{ll}
    \mbox{minimize}   & \sum_{i, j} X_{i, j} \log {X_{i, j}}/{C_{i, j}}\\
    \mbox{subject to} & X \; \mbox{doubly stochastic} \\
                      & X_{i,j} = 0 \; \forall \; (i,j) \; \mbox{with} \; C_{i,j} = 0,
 \end{array}
\end{equation}
where we define $0 \cdot log(0) = 0$, $0 \cdot log(0/0) = 0$, and $1 \cdot log(1/0) = \infty$. Note, however, that the relative entropy is not a proper distance metric since, it is not symmetric, does not satisfy the triangular inequality and can take the value $\infty$.

The \emph{matrix balancing problem} can be solved by iterative methods with remarkable scalability. Standard and simple iterative methods exist that exhibit linear per-iteration complexity w.r.t.\ the number of nonzeros in $C$ and linear convergence rate \citep{Knight2008}, \citep{Idel2016}. More recent algorithms can exhibit a super-linear convergence rate and increased robustness in many practical situations \citep{Knight2013}.

The aim of the paper is to show that the direct minimization of a least squares objective in doubly stochastic approximation, which has a long history dating back to the influential paper of \citep{Deming1940}\footnote{\cite{Deming1940} consider the weighted least squares cost $\frac{1}{2}\sum_{i, j}[X_{i, j} - C_{i, j}]^2/C_{i, j}$ which we treat in \S \ref{subsubsec:generalizations}.}, can also by solved efficiently. This gives practitioners a new solution to a very important problem of doubly stochastic matrix approximation,  which might prove useful for cases where the relative entropy metric is not suitable to their problem.

The approach we present is also flexible in the sense that it can handle other interesting cases, for example, where $C$ is non-symmetric or rectangular, $\norm{\cdot}$ is a weighted Frobenius norm, and $X \mathbf{1}$ and $X^T \mathbf{1}$ are required to sum to an arbitrary, given vector. We discuss these generalizations in \S \ref{subsubsec:generalizations}.

\paragraph{Related work} \cite{Zass2007} consider the problem \ref{eqn:approximation_problem} in the case when $C$ is fully dense. They suggest an alternating projections algorithm that has linear per-iteration complexity. The approach of \cite{Zass2007} resembles the results of \S \ref{subsubsec:generalizations}, but as we will see in the experimental section, it is not guaranteed to converge to an optimizer. 

The paper is organized as follows: In Section \ref{sec:modeling}, we introduce a series of reformulations to Problem \ref{eqn:matrix_balancing}, resulting in a problem that is much easier to solve.In Section \ref{sec:admm}, we suggest a  solution method that reveals a particular structure in the problem. Finally, in Section \ref{sec:results}, we present a series of numerical results that highlight the scalability of the approach.

\paragraph{Notation used} The symbol $\otimes$ denotes the Kronecker product, $\vc(\cdot)$ the operator that stacks a matrix into a vector in a column-wise fashion and $\mat(\cdot)$ the inverse operator to $\vc(\cdot)$ (see \citep[12.3.4]{Golub2013} for a rigorous definition). Given two matrices, or vectors, with equal dimensions $\odot$ denotes the Hadamard (element-wise) product. $\norm{\cdot}$ denotes the 2-norm of a vector and the Frobenius norm of a matrix, while $\mathrm{card}(\cdot)$ the number of nonzero elements in a vector or a matrix. Finally $\epsilon_{\text{p}}$ denotes the machine precision.

\section{MODELING \ref{eqn:approximation_problem} EFFICIENTLY} \label{sec:modeling}
In this section, we present a reformulation of the doubly stochastic approximation problem \ref{eqn:approximation_problem} suitable for solving large-scale problems. One of the difficulties with the original formulation, \ref{eqn:approximation_problem}, is that it has $n^2$ variables and $2n^2 + 2n$ constraints. Attempting to solve \ref{eqn:approximation_problem} with an off-the-shelf QP solver, such as Gurobi \citep{gurobi}, can result in out-of-memory issues just by representing the problem's data even for medium-sized problems.

In order to avoid this issue, we will perform a series of reformulations that will eliminate variables from \ref{eqn:approximation_problem}. The final problem, i.e.,~\ref{eqn:approximation_problem_symm_reduced}, will have significantly fewer variables and constraints while maintaining a remarkable degree of structure, which will be revealed and exploited in the next section.

We first take the obvious step of eliminating all variables in \ref{eqn:approximation_problem} that are constrained to be zero, as prescribed by the constraint
\begin{equation}
  X_{i,j} = 0 \; \mbox{for all} \; (i,j) \; \mbox{such that} \; C_{i,j} = 0.
\end{equation}
Indeed, consider $\vc(\cdot)$, the operator that stacks a matrix into a vector in a column-wise fashion and $H$ an $n_{\textrm{nz}} \times n^2$ matrix, where $n_{\textrm{nz}} \eqdef \text{card}(C)$, that selects all the nonzero elements of $\vc(C)$. Note that $H$ depends on the sparsity pattern $S$ of $C$, defined as the $0-1$ $n \times n$ matrix
\begin{equation}
  S_{i,j} \eqdef
  \begin{cases}
    0 & C_{i,j} = 0 \\
    1 & \text{otherwise},
  \end{cases}
\end{equation}
but we will not denote this dependence explicitly. We can now isolate the nonzero variables contained in $X$ and $C$ in $n_{\textrm{nz}}-$dimensional vectors defined as
\begin{equation}
  x \eqdef H \vc(X), \quad c \eqdef H \vc(C).
\end{equation}
Note that $x$ is simply a re-writing of $X$ in $\mathbb{R}^{n_{\textrm{nz}}}$, since for any $X \in \mathcal{S} \eqdef \set{X \in \mathbb{R}^{n \times n}}{X_{i,j} = 0 \; \mbox{for all} \; C_{i,j} = 0}$ we have
\begin{equation*}
  H^T H \vc(X) = \vc(X) \Leftrightarrow H^T x = \vc(X),
\end{equation*}
due to the fact that $H^T H = \diag(\vc(S))$
Thus every $x$ defines a unique $X \in \mathcal{S}$ and vice versa.

We can now describe the constraints of \ref{eqn:approximation_problem}, i.e.,~$X \geq 0$, $X\mathbf{1}_n = \mathbf{1}_n$ and $X^T\mathbf{1}_n = \mathbf{1}_n$, on the ``$x-$space''. Obviously $X \geq 0$ trivially maps to $x \geq 0$. Furthermore, recalling the standard Kronecker product identity
\begin{equation} \label{eqn:kronecker_identity}
  \vc(LMN) = (N^T \otimes L)\vc(M)
\end{equation}
 for matrices of compatible dimension, the constraints $X \mathbf{1}_n = \mathbf{1}_n$ and $X^T \mathbf{1}_n = \mathbf{1}_n$ of \ref{eqn:approximation_problem} can be rewritten in an equivalent vectorized form as
\begin{equation} \label{eqn:vectorized_constraints}
  \begin{bmatrix}
  \mathbf{1}_n^T \otimes I_n \\
  I_n \otimes \mathbf{1}_n^T
  \end{bmatrix}
  \vc(X) = \mathbf{1}_{2n}
\end{equation}
or, by noting that $H ^T x = \vc(X)$, as
\begin{equation} \label{eqn:constraints_in_x}
  \begin{bmatrix}
  \mathbf{1}_n^T \otimes I_n \\
  I_n \otimes \mathbf{1}_n^T
  \end{bmatrix}
  H^T x = \mathbf{1}_{2n}.
\end{equation}
Thus \ref{eqn:approximation_problem} can be rewritten in the ``$x-$space'' as
\begin{equation}\label{eqn:approximation_problem_reduced} \tag{$\mathcal{P}_1$}
  \begin{array}{ll}
    \mbox{minimize}   & \frac{1}{2}\norm{x - c}^2 \\
    \mbox{subject to} & x \geq 0 \\
    & \begin{bmatrix}
  \mathbf{1}_n^T \otimes I_n \\
  I_n \otimes \mathbf{1}_n^T
  \end{bmatrix}
  H^T x = \mathbf{1}_{2n}.
 \end{array}
\end{equation}

A further reduction of the variables and constraints of \ref{eqn:approximation_problem_reduced} can be achieved by exploiting the symmetry of $C$. To this end, note that when $C$ is symmetric the optimal doubly stochastic approximation will also be symmetric according to the following proposition:
  \begin{proposition} \label{prop:symmetry}
    If $C$ is symmetric, then the optimal solution $X^*$ of \ref{eqn:approximation_problem} is also symmetric.
  \end{proposition}
  \begin{proof}
    Assume the contrary, so that $X^*$ is optimal but asymmetric.   Then the matrix ${X^*}^T$ is a feasible solution for \ref{eqn:approximation_problem} since it remains element-wise negative when its row and column sums are exchanged, and has an identical objective value since $C$ is assumed symmetric.   Then define the symmetric matrix $\tilde X$ as the convex combination
    \[
    \tilde X := \frac{1}{2}(X^* + {X^*}^T)
    \]
    which is also a feasible point for \ref{eqn:approximation_problem}. Since the objective function is strictly convex (at least on the subset of elements of X not constrained to be zero), the objective function evaluated at $\tilde X$ will be strictly lower than that for both $X^*$, contradicting the optimality of $X^*$. \qed
  \end{proof}
  It follows that restricting the feasible set of \ref{eqn:approximation_problem} to symmetric matrices does not affect its solution. We will exploit this by eliminating all the  variables embedded into $X$ that are below its main diagonal. Define an upper triangular matrix $X_u$ consisting of scaled elements of $X$ such that $X = X_u+X_u^T$, and likewise for $C$, i.e.
  \begin{equation} \label{eqn:x_upper_triangular}
    X_u \eqdef U \odot X, \quad C_u \eqdef U \odot C
  \end{equation}
where
\begin{equation}
    U \eqdef
    \begin{bmatrix}
      \frac{1}{2} & 1 & \cdots &  1 \\
      & \ddots  & & \vdots \\
       & & \frac{1}{2} & 1 \\
      0 & & & \frac{1}{2}
    \end{bmatrix}.
  \end{equation}
  As in the previous definitions, define $H_u$ as the matrix that stacks all the nonzeros of $C_u$ in an column-wise fashion, which is used to extract the nonzero elements of $X_u$ and $C_u$, i.e.,~scaled nonzero elements of the upper triangular part of $X$,  to the vectors
  \begin{equation} \label{eqn:upper_triangular_vectors}
    x_u \eqdef H_u\vc(X_u), \quad c_u \eqdef H_u \vc(X_u).
  \end{equation}

  We can now write down our reduced optimization problem. Note that, although it might not be directly evident, the following problem possesses a remarkable degree of internal structure that is exploited in the suggested solution algorithm of the next section.
  \begin{theorem} \label{thm:final_problem}
    Solving \ref{eqn:approximation_problem} for a symmetric $C$ is equivalent to solving
    \begin{equation}\label{eqn:approximation_problem_symm_reduced} \tag{$\mathcal{P}_2$}
      \begin{array}{ll}
        \mbox{minimize}   & \frac{1}{2}\norm{p\odot(x_u - c_u)}^2\\
        \mbox{subject to} & x_u \geq 0 \\
        & A x_u = \mathbf{1}_{n},
      \end{array}
    \end{equation}
    where
    \begin{equation*}
      p \eqdef H_u\vc\left(\begin{bmatrix}
        2 & \sqrt{2} & \cdots & \sqrt{2} \\
        & \ddots  & & \vdots \\
         & & 2 & \sqrt{2} \\
        0 & & & 2
      \end{bmatrix}\right)
    \end{equation*}
    \begin{equation*}
    A_1 \eqdef
    \mathbf{1}_{n}^T \otimes I_n, \quad
    A_2 \eqdef
    I_n \otimes \mathbf{1}_{n}^T
  \end{equation*}
  and $A \eqdef (A_1 + A_2) H_u^T$,
  in the sense that \ref{eqn:approximation_problem} is feasible iff \ref{eqn:approximation_problem_symm_reduced} is, and the optimizer $X^*$ of \ref{eqn:approximation_problem} can be constructed from the optimizer $x_u^*$ of \ref{eqn:approximation_problem_symm_reduced} using \eqref{eqn:upper_triangular_vectors} and \eqref{eqn:x_upper_triangular}.
  \end{theorem}
  \begin{proof}
  We will first show that every feasible $x_u$ of \ref{eqn:approximation_problem_symm_reduced} defines a feasible $X$ 
  for \ref{eqn:approximation_problem}, where $\vc(X_u) \eqdef H_u^T x_u$, with the same objective value.
  Similarly to $S$, define $S_u \in \mathbb{R}^{n\times n}$ as an $0-1$ matrix that represents the sparsity of $C$ and the upper triangular of $C$ respectively, i.e.
  \begin{equation*}
    S_{u_{i,j}} =
    \begin{cases}
      0 & C_{i,j} = 0, \text{ or } i < j\\
      1 & \text{otherwise}.
    \end{cases}
  \end{equation*}

  The equality of the objective value can be shown as follows:
  \begin{equation}
  \begin{aligned}
     &\norm{p\odot (x_u - c_u)} \\
    = &\norm{(\mathbf{1}_{n \times n}\sqrt{2} + (2 -\sqrt{2})I_{n}) \odot (X_u - C_u)} \\
    = &\norm{(\mathbf{1}_{n \times n}\sqrt{2} + (2 -\sqrt{2})I_{n}) \odot U \odot (X - C)} \\
    = &\norm{(\mathbf{1}_{n \times n}\sqrt{2} + (1 -\sqrt{2})I_{n}) \odot S_U \odot (X - C)} \\
    = &\norm{S \odot (X - C)} = \norm{X - C}.
  \end{aligned}\label{eqn:objective}
  \end{equation}
  Furthermore, similarly to \eqref{eqn:kronecker_identity}-\eqref{eqn:constraints_in_x}, we have
  \begin{equation}
    \begin{aligned}
    X\mathbf{1}_n &= (X_u + X_u^T)\mathbf{1}_n \\
    &= (\mathbf{1}_{n}^T \otimes I_n + I_n \otimes \mathbf{1}_{n}^T) \vc(X_u) \\
    &= (\mathbf{1}_{n}^T \otimes I_n + I_n \otimes \mathbf{1}_{n}^T) H_u^T x_u
    = A x_u,
  \end{aligned}
  \label{eqn:feasibility}
\end{equation}
  resulting in $X\mathbf{1}_n = \mathbf{1}_n$ due to the feasibility of $x_u$ for \ref{eqn:approximation_problem_symm_reduced}. Due to the symmetry of $X$ we also get $X^T\mathbf{1}_n$. Finally, $X$ is nonnegative construction  and has sparsity pattern $S$. Therefore $X$ is feasible for \ref{eqn:approximation_problem}.

  Likewise, following \eqref{eqn:objective}-\eqref{eqn:feasibility} in reverse order, we can show that every symmetric feasible matrix $X$ of \ref{eqn:approximation_problem} defines an $x_u \eqdef H_u \vc(X_u)$, where $ X_u \eqdef U \odot X$, that is feasible for \ref{eqn:approximation_problem_symm_reduced} and has identical objective value. Since only symmetric optimizers exist for \ref{eqn:approximation_problem} (Lemma \ref{prop:symmetry}) this concludes the proof. \qed
  \end{proof}
  Unlike \ref{eqn:approximation_problem} which has $n^2$ and $2n$ constraints, \ref{eqn:approximation_problem_symm_reduced} has approximately $n_{\textrm{nz}}/2$ and $n$ constraints. Furthermore, it possesses a specific internal structure that we exploit in the solution algorithm presented in \S\ref{subsec:algorithm}.
  \section{SOLUTION METHOD} \label{sec:admm}
  In this section, we describe how the reduced problem \ref{eqn:approximation_problem_symm_reduced} can be solved with ADMM. We begin with a brief introduction to the ADMM algorithm in the general setting, which follows \citep{Boyd2011}, and then describe how ADMM can be applied efficiently for \ref{eqn:approximation_problem_symm_reduced}.

  \subsection{The Alternating Direction Method of Multipliers (ADMM) } \label{subsec:admm}
  Several optimization problems, including reformulations of \ref{eqn:approximation_problem_symm_reduced} \citep{Stellato2017}, are concerned with the minimization of a function $q$ that can be decomposed into two parts $q = f + g$ such that optimizing independently $f$ or $g$ is tractable. 
  Ideally, if $f$ and $g$ operate on disjoint variables, i.e.,~if $q(\chi, \psi) = f(\chi) + g(\psi)$, then $q$ can also be optimized efficiently by merely minimizing $q$ over $\chi$ and $\psi$ independently. However, it is often that case that there is some coupling between $\chi$ and $\psi$ which we assume to be in the form $A\chi + B \psi = d$, resulting in the following optimization problem
  \begin{equation}\label{eqn:splitting}
    \begin{array}{ll}
      \mbox{minimize}   & f(\chi) + g(\psi) \\
      \mbox{subject to} & A\chi + B\psi = d
   \end{array}
  \end{equation}
  where $\chi, \psi$ denote the decision variables, $f$, $g$ are proper lower-semicontinuous convex functions, and $A, B, d$ are matrices of appropriate dimensions.

  The Alternating Direction Method of Multipliers (ADMM) is a first-order (i.e.,~``gradient-based'') algorithm that solves \eqref{eqn:splitting} while exploiting the assumption that $f$ and $g$ can be easily optimized independently. ADMM is very closely related to performing (sub)gradient ascent to the dual of \eqref{eqn:splitting}, i.e.,~to the problem
  \begin{equation*}
    \begin{array}{ll}
      \mbox{maximize} & h(y) \eqdef \inf_{\chi, \psi}L(\chi, \psi, y),
    \end{array}
  \end{equation*}
  where $L(\chi, \psi, y) \eqdef f(\chi) + h(\psi) + y^T(A\chi + B\psi - d)$ is the Lagrangian of \eqref{eqn:splitting}, and $y$ is a Lagrange multiplier associated with the equality constraint of \eqref{eqn:splitting}. Thus dual ascent in merely the following iteration
  \begin{equation}
    y^{k+1} = y^k + \rho \nabla_y h(y^k),
  \end{equation}
  where $\rho > 0$ is a step size, and $\nabla_y h(y)$ is a subgradient of $h$ which can be evaluated as \citep[\S 2.1]{Boyd2011}
  \begin{equation*}
    \nabla_y h(y^{k+1}) = A\chi^{k+1} + B\psi^{k+1} - d,
  \end{equation*}
  where $(\chi^{k+1}, \psi^{k+1}) \eqdef \inf_{\chi, \psi}L(\chi, \psi, y^k)$,
  resulting in the following iterative scheme:
  \begin{align*}
    (\chi^{k+1}, \psi^{k+1}) &= \inf_{\chi, \psi}L(\chi, \psi, y^k)\\
    y^{k+1} &= y^{k} + \rho(A\chi^{k+1} + B\psi^{k+1} - d)
  \end{align*}
  One of the issues with dual ascent, apart from its weak convergence guarantees, is that $(\chi^{k+1}, \psi^{k+1}) = \inf_{\chi, \psi}L(\chi, \psi, y^k)$ might not be easily computable as it does not exploit the fact that $f$ and $g$ are easily optimized independently.
  To overcome this difficulty, ADMM \emph{approximates} the minimization of the Lagrangian over $(\chi^{k+1}, \psi^{k+1})$ with minimization first over $\chi^{k+1}$ and then $\psi^{k+1}$, independently. Separating the minimization over $\chi$ and $\psi$ into the above two steps is precisely what allows for the tractability of ADMM. Furthermore, ADMM introduces a penalty term to the Lagrangian $L$ resulting in the augmented Lagrangian
  \begin{align*}
    L_{\rho}(\chi, \psi, y) \eqdef f(\chi) + h(\psi) + y^T(A\chi + B\psi - d) \\ + (\rho/2)\norm{A\chi + B\psi - d}^2,
  \end{align*}
  which extends the convergence properties of ADMM.
  The resulting algorithm iterates as
  \begin{align*}
    \chi^{k+1} &= \inf_{\chi}L_{\rho}(\chi, \psi^k, y^k) \\
    \psi^{k+1} &= \inf_{\psi}L_{\rho}(\chi^{k+1}, \psi, y^k) \\
    y^{k+1} &= y^{k} + \rho(A\chi^{k+1} + B\psi^{k+1} - d)
  \end{align*}
  It is remarkable that, although ADMM can be considered as approximate dual ascent, it has superior convergence guarantees \citep{Eckstein1992}.

  \subsection{Solving \ref{eqn:approximation_problem_symm_reduced} with ADMM} \label{subsec:algorithm}
  Problem \ref{eqn:approximation_problem_symm_reduced} is a quadratic program (QP) with strictly convex objective function. Solving QPs with ADMM has been widely studied in the literature \citep{Stellato2017}, \citep{Garstka2019}, \citep{ODonoghue2016}. We will follow the approach of \citep{Stellato2017} which can solve \ref{eqn:approximation_problem_symm_reduced} by applying ADMM to the following splitting
  \begin{equation}\label{eqn:splitting_osqp}
    \begin{array}{ll}
      \mbox{minimize} & f(\tilde x, \tilde z) + g(x, z) \\
      \mbox{subject to} & (\tilde x, \tilde z) = (x, z)
    \end{array}
  \end{equation}
  where $f$ and $g$, are defined as
  \begin{align*}
  f(\tilde x, \tilde z) &= \frac{1}{2}x^T P x - Pc_u + \mathcal{I}_{A\tilde x = \tilde z}(\tilde x, \tilde z) \\
  g(x, z) &= \mathcal{I}_{x \geq 0}(x) + \mathcal{I}_{z = \mathbf{1}_{2n}}(z)
  \end{align*}
  and $P \eqdef \diag(p \odot p)$, $\mathcal{I}_{A\tilde x = \tilde z}$, $\bar n_{\textrm{nz}}$ are the number of nonzeros in the upper triangular of $C$ and $\mathcal{I}_{x \geq 0}$, $\mathcal{I}_{z = \mathbf{1}_{2n}}$ denote the indicator functions of the sets $\set{(x, z) \in \mathbb{R}^{\bar n_{\textrm{nz}}} \times \mathbb{R}^{2n}}{Ax = z}$, $\set{x \in \mathbb{R}^{\bar n_{\textrm{nz}}}}{x \geq 0}$ and $\set{x \in \mathbb{R}^{2n}}{x = \mathbf{1}_{2n}}$ respectively.

  Applying ADMM for the problem \eqref{eqn:splitting_osqp} results\footnote{Algorithm \ref{alg:admm} includes two extensions over the simple ADMM algorithm discussed in \S \ref{subsec:admm}: it uses \emph{over-relaxation}, a commonly used variation that can increase the speed of convergence \citep[Figure 2]{Eckstein1992}, and two different step sizes $\rho, \sigma$ for the update of Lagrange multipliers $w$ and $y$ respectively.} in Algorithm \ref{alg:admm}, where $\Pi_{+}$ denotes the projection of a vector to the nonnegative orthant and $\Pi_1(x) \eqdef \mathbf{1}$ for any vector $x$. Refer to \citep[\S 3]{Stellato2017} for details.
  \begin{algorithm}
    \textbf{given} initial values $x^0, z^0, y^0$ and parameters $\rho > 0$,  $\sigma > 0$, and $\alpha \in (0, 2)$\;
    \Repeat{termination condition is satisfied}{
    $(\tilde x^{k+1}, \tilde z^{k+1}) \leftarrow$  \text{solution of the linear system}
    $
      \begin{bmatrix}
        (P + \sigma I_{\bar n_{\textrm{nz}}})& \rho A^T \\
        \rho A & -\rho I_{2n}
      \end{bmatrix}
      \begin{bmatrix}
        \tilde x^{k+1}\\
        \tilde z^{k+1}
      \end{bmatrix}
      =
      \begin{bmatrix}
        \sigma x^{k} - w^k + A^T (\rho z^k - y^k) + P c_u\\
        0
      \end{bmatrix}
    $\;
    $x^{k+1} \leftarrow \Pi_{+}(\alpha \tilde x^{k+1} + (1 - \alpha) x^k$)\;
    $z^{k+1} \leftarrow \Pi_{\mathbf{1}}(\alpha \tilde z^{k+1} + (1 - \alpha)z^k + \rho^{-1}y^k)$\;
    $w^{k+1} \leftarrow w^k + \sigma(\alpha \tilde x^{k+1} + (1 - \alpha)x^k - x^{k+1})$\;
    $y^{k+1} \leftarrow y^k + \rho(\alpha \tilde z^{k+1} + (1 - \alpha)z^k - z^{k+1})$\;
    }
    \caption{Solving \ref{eqn:approximation_problem_symm_reduced} with ADMM}
    \label{alg:admm}
  \end{algorithm}

The most computationally intensive operation of Algorithm \ref{alg:admm} is the solution of the $(\bar n_{\textrm{nz}} + n)\times (\bar n_{\textrm{nz}} + n)$ linear system in line 3. We will show that its solution can be obtained by solving instead a reduced $n\times n$ linear system.
\begin{fact} \label{fact:equality_qp}
  Consider the following linear system
  \begin{equation} \label{eqn:linear_system}
    \begin{bmatrix}
      P + \sigma I_{\bar n_{\textrm{nz}}} & \rho A^T \\
      \rho A & -\rho I_{2n}
    \end{bmatrix}
    \begin{bmatrix}
      x\\
      z
    \end{bmatrix}
    =
    \begin{bmatrix}
      r \\
      0
    \end{bmatrix}
  \end{equation}
  where $x, r \in \mathbb{R}^{\bar n_{\textrm{nz}}}, z \in \mathbb{R}^{\bar n_{\textrm{nz}}}$ and $A \in \mathbb{R}^{2n \times \bar n_{\textrm{nz}}}$. Its solution can be obtained as follows
  \begin{enumerate}
    \item Obtain $z$ by solving the following $n \times n$ positive definite linear system
    \begin{equation} \label{eqn:reduced_linear_system}
      (\rho A(P + \sigma I_{\bar n_{\textrm{nz}}})^{-1}A^T + I_{n}) z = A(P+\sigma I_{\bar n_{\textrm{nz}}})^{-1}r
    \end{equation}
    \item Obtain $x$ as $(P + \sigma I_{\bar n_{\textrm{nz}}})^{-1}(r - \rho A^Tz)$.
  \end{enumerate}
\end{fact}
\begin{proof}
  The first block row gives $(P + \sigma I_{\bar n_{\textrm{nz}}}) x + \rho A^Tz = r \Leftrightarrow x = (P + \sigma I_{\bar n_{\textrm{nz}}})^{-1}(r - \rho A^T z)$. Reducing the variable $x$ from \eqref{eqn:linear_system} results in \eqref{eqn:reduced_linear_system}. \qed
\end{proof}
Thus solving \eqref{eqn:linear_system} can be reduced to solving a linear system with left hand side
\begin{equation} \label{eqn:reduced_matrix}
  (\rho A(P + \sigma I_{\bar n_{\textrm{nz}}})^{-1}A^T + I_{n}).
\end{equation}
Fortunately, the reduced matrix \eqref{eqn:reduced_matrix}, which is equal to
$A\diag (H_u\vc(U))A^T$
with
\begin{equation*}
U \eqdef
\begin{bmatrix}
  (4 + \sigma)^{-1} & (2 + \sigma)^{-1} & \cdots & (2 + \sigma)^{-1} \\
  & \ddots  & & \vdots \\
   & & (4 + \sigma)^{-1} & (2 + \sigma)^{-1} \\
  0 & & & (4 + \sigma)^{-1}
\end{bmatrix}, \\
\end{equation*}
turns out to be positive definite with a sparsity pattern matching that of $C+I$:
\begin{theorem} \label{thm:AA^T}
  The following relation holds
  \begin{align*} \label{eqn:AA^T}
    &A\diag(H_u\vc(D))A^T \\
    = &S \odot (D + D^T) + \diag(S \odot (D + D^T) \mathbf{1})
  \end{align*}
  for any upper triangular $n\times n$ matrix $D$.
\end{theorem}
\begin{proof}
  The proof is in the Appendix.\qed
\end{proof}

The solution of the reduced linear system \eqref{eqn:reduced_linear_system} can be obtained given an initial Cholesky factorization of $\rho A(P + \sigma I_{\bar n_{\textrm{nz}}})^{-1}A^T + I_n$, or even with a factorization free algorithm e.g.,~Conjugate Gradients which primarily consists of repeated matrix-vector multiplications with $\rho A(P + \sigma I_{\bar n_{\textrm{nz}}})^{-1}A^T + I_n$. The fact that the linear system to be solved has the same sparsity pattern of $S + I$ can be particularly beneficial in cases where a fill-in reducing permutation is already known for the matrix under approximation $C$ (and thus for $S$), since the same permutation could be used before the factorization of $\rho A(P + \sigma I_{\bar n_{\textrm{nz}}})^{-1}A^T + I_n$ resulting in reduced fill-in and thus significant speedup.

\subsubsection{Convergence and Feasibility}
We terminate Algorithm \ref{alg:admm} when the primal and dual residuals of \ref{eqn:approximation_problem_symm_reduced}
\begin{align*}
  r_{\text{prim}} &\eqdef \norm{Au - \mathbf{1}}_\infty, \\
  r_{\text{dual}} &\eqdef \norm{P x - c_u + A^T y + w}_\infty
\end{align*}
become smaller than some acceptable tolerance. Algorithm \ref{alg:admm} is guaranteed to converge to the solution of \ref{eqn:approximation_problem_symm_reduced} whenever \ref{eqn:approximation_problem_symm_reduced}, or equivalently \ref{eqn:approximation_problem}, is feasible.

We next establish conditions that characterize the feasibility of \ref{eqn:approximation_problem}:
\begin{lemma} \label{lem:feasibility}
  Problem \ref{eqn:approximation_problem} is feasible if and only if there exists a set of indices $\mathcal{I} = \{(i_1, j_1) \dots, (i_n, i_n)\}$ corresponding to exactly one nonzero element from each row and column of $C$.
\end{lemma}
\begin{proof}
  Regarding the ``if'' part, the $n \times n$ matrix
  \begin{equation}
    X_{i, j} = \begin{cases}
      1 \quad \text{if} \quad i, j \in I \\
      0 \quad \text{otherwise}
    \end{cases}
  \end{equation}
  is a feasible point for \ref{eqn:approximation_problem}.
  Regarding the ``only if'' part, since \ref{eqn:approximation_problem} has a feasible point $X$, then according to \citep[Theorem 1]{Perfect1965} there exists a set of indices $\mathcal{I} = \{(i_1, j_1) \dots, (i_n, i_n)\}$ corresponding to exactly one nonzero element from each row and column of $X$. Due to the second constraint of \ref{eqn:approximation_problem}, the same argument also holds for $C$.\qed
\end{proof}
Note that Lemma \ref{lem:feasibility} and \citep[Theorem 1]{Perfect1965} imply that \ref{eqn:approximation_problem} is feasible whenever the matrix balancing problem is feasible, i.e. when the matrix $C$ has total support \citep{Knight2013}.

\subsubsection{Special cases and generalizations} \label{subsubsec:generalizations}
\paragraph{The case where $C$ is almost or fully dense:}
In the special case where $C$ is fully dense we have $S = \mathbf{1}_n\mathbf{1}_n^T$, and Theorem \ref{thm:AA^T} gives
\begin{equation*}
  \rho A(P + \sigma I_{\bar n_{\textrm{nz}}})^{-1}A^T + I_n = \alpha I +  \beta \mathbf{1}_{n \times n}
\end{equation*}
where
$$
  \alpha \eqdef \frac{\sigma \rho}{(2 + \sigma/2)(2 + \sigma)} + \frac{n \rho}{2 + \sigma} + 1
  \quad \text{ and } \quad
  \beta \eqdef \frac{\rho}{2 + \sigma}.
$$
Using the Sherman-Morrison formula, we can explicitly calculate the inverse of \eqref{eqn:reduced_linear_system} as
\begin{equation}
  \label{eqn:explicit_inverse}
  (\rho A(P + \sigma I_{\bar n_{\textrm{nz}}})^{-1}A^T + I_n)^{-1} = \frac{1}{\alpha}\left(I - \frac{\beta  \mathbf{1}_n \mathbf{1}_n ^T}{\alpha + \beta n}\right).
\end{equation}
We can then solve \eqref{eqn:linear_system} and perform ADMM on \ref{eqn:approximation_problem_symm_reduced} without the need to perform an initial matrix factorization.

This approach can also be extended to cases where $C$ has a relatively small number of zero elements. Indeed, by avoiding the elimination of the zero variables of \ref{eqn:approximation_problem} we get the following variant of \ref{eqn:approximation_problem_symm_reduced}:
  \begin{equation} \label{eqn:approximation_problem_dense} \tag{$\mathcal{P}_3$}
    \begin{array}{ll}
      \mbox{minimize}   & \frac{1}{2}\norm{p \odot x_u - c_u}^2 \\
      \mbox{subject to} & x_u \geq 0 \\
      & x_{u_i} = 0 \; \forall \; i \; \mbox{with} \; c_{u_i} = 0\\
      & Ax_u = \mathbf{1}_{2n}
   \end{array}
  \end{equation}
  where $p, A, x_u$, and $c_u$ are defined according to Section \ref{sec:modeling}, but with $H_u$ an $\frac{(n + 1)n}{2} \times n^2$ linear map that extracts \emph{all} the upper triangular elements of a vectorized $n \times n$ matrix. \ref{eqn:approximation_problem_dense} can then be solved with Algorithm \ref{alg:admm}, with the following two changes.
  First, we replace $\Pi_{+}$ (line 2, Algorithm \ref{alg:admm}) with $\Pi_{\mathcal{S}}$, where $\mathcal{S} \eqdef \{x \in \mathbb{R}^{(n^2 +n)/2} \mid x \geq 0 \text{ and } x_i = 0 \text{ for all } i \text{ such that } c_i = 0\}$. Secondly, the solution of the linear system of line 3 of Algorithm \ref{alg:admm} is trivially solved using Fact \ref{fact:equality_qp} and \eqref{eqn:explicit_inverse}.

  \paragraph{Solving variants of \ref{eqn:approximation_problem} with Algorithm \ref{alg:admm}:}
  Algorithm \ref{alg:admm} can be easily adjusted to the case where $\norm{\cdot}$ in the objective of \ref{eqn:approximation_problem} is a weighted Frobenius norm, i.e. $\norm{X} = \norm{W \odot X}_F$ where $W$ is a given symmetric matrix.
  The only thing that has to change is the definition of $p$, and thus of $P \eqdef \diag(p \odot p)$, to:
  \begin{equation*}
    p \eqdef H_u\vc\left(\begin{bmatrix}
      2 & \sqrt{2} & \cdots & \sqrt{2} \\
      & \ddots  & & \vdots \\
        & & 2 & \sqrt{2} \\
      0 & & & 2
    \end{bmatrix} \odot W \right).
  \end{equation*}
  Theorem 3.2 could then be used to solve the linear system of Algorithm \ref{alg:admm} (line 3) efficiently.
  Similarly, we can allow for general constraints $X \mathbf{1} = r$ and $X^T \mathbf{1} = r$ in \ref{eqn:approximation_problem}, where $r$ is a given nonnegative vector, by simply changing $\Pi_\mathbf{1}$ in line 5 of Algorithm \ref{alg:admm} to $\Pi_r(x) \eqdef r$. Finally, non-square or non-symmetric matrices $C$ can be solved via use of \ref{eqn:approximation_problem} for the symmetric matrix $\begin{bmatrix} 0 & C \\ C^T & 0 \end{bmatrix}$.

\section{NUMERICAL EXPERIMENTS} \label{sec:results}
In this section we present numerical results of Algorithm \ref{alg:admm} on a range of matrix normalization problems. We provide a \texttt{Julia} implementation of the Algorithm, along with code that generates all the results of this section at:
\centerline{\texttt{github.com/oxfordcontrol/DoublyStochastic.jl}}\\

\subsection{Normalizing Hi-C Contact Matrices of the Human Genome in the 3D Space}
We first present results on the application of our method to real-world contact matrices describing the 3D structure of the human genome, starting with a description of the nature of these matrices. The reader can also find an excellent, visual introduction of the problem in \texttt{youtu.be/dES-ozV65u4}. The human genome has an end-to-end length on the order of meters when unfolded, but fits inside the cell nucleus with dimensions on the order of micrometers, implying that the genome is heavily folded in the 3D space. The 3D structure of the genome can be examined by breaking the genome into a number of pieces and measuring how many contacts exist between each piece in the 3D space \citep{Rao2014}.
This produces \emph{Hi-C contact matrices}, where the term \emph{Hi-C} describes the particular experimental procedure used.

The process is, however, subject to errors and experimental constraints. To alleviate these issues, the contact matrix is normalized so that all its rows and columns sum to the same value. The matrix balancing approach is often the method of choice for this task \citep[Supplemental Material II.b]{Rao2014}, but other methods have also been suggested in the literature \citep{Yaffe2011}.

We will show that our approach can also be used for this task even for contact matrices containing hundreds of millions of nonzero entries as in \cite{Rao2014}. In particular, we consider normalization of the contact matrix of the $7^{\textrm{th}}$ chromosome of the GM12878 cell\footnote{The data corresponding to the thresholding criterion $\text{MAPQ} \geq 30$ were used \citep[Supplemental Material IIa.4]{Rao2014}. Obtained from the GM12878 ``combined'' intrachromosomal tarball at \footnotesize{\texttt{ncbi.nlm.nih.gov/geo/query/acc.cgi?acc=GSE63525}}}, thus replicating \citep[Figure 1.C]{Rao2014}.
Since the genome consists of sequences of the bases \emph{adenine, guanine, cytosine} and \emph{thymine}, it is typical to measure the length of each genome piece by the average number of bases it contains. The total range of the contact matrices is $0$ to $160$ mega-bases (Mb). We consider two discretization lengths, $1$ kilo-base (Kb), and $5$Kb, which result in contact matrices of $151$ and $82$ million nonzeros respectively. Figure \ref{fig:hi-c} provides detailed views of the contact matrices, spanning the range $[137.2$--$137.8\text{Mb}]$ for the contact map at $5$Kb resolution and the $[137.55$--$137.75\text{Mb}]$ for the one at $1$Kb resolution. These regions were chosen to highlight interesting regions of the contact matrix as they appear in \citep[Figure 1.C, rightmost column, two bottom subfigures]{Rao2014}.

Our approach produces considerably different normalized contact matrices than the matrix balancing approach. In particular, our approach results in contact matrices that have increased sparsity and higher contrast. This is unlike the matrix balancing approach which results in a normalized matrix that has exactly the same nonzeros as the original matrix. Although further investigation is necessary for the evaluation of the suitability of the approach in Hi-C data, the results indicate that our method can be used to normalize very large Hi-C datasets leading to promising visual results.
\begin{figure}[b!]
  \fbox{\includegraphics[width=1.0\linewidth]{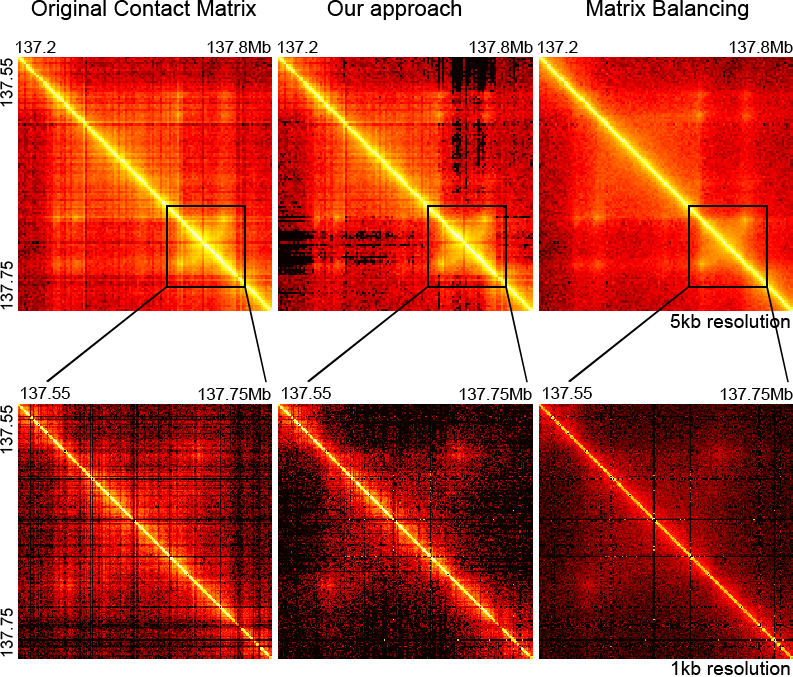}}
  \caption{Details of Hi-C Contact Matrices for the $7^{\textrm{th}}$ chromosome of the GM12878 cell, corresponding to \citep[Figure 1.C, rightmost column, two bottom subfigures]{Rao2014}. Top row shows the area $[137.2 -137.8 \text{Mb}]^2$ for the contact matrix of $5$Kb resolution. Bottom row shows the area $[137.55-137.75 \text{Mb}]^2$ for the contact matrix of $1$Kb resolution. Areas representing zero contacts are depicted in black, while areas with a high number of contacts are shown in yellow. The total area of the contact matrices is $[0-160\text{Mb}]^2$, thus the areas depicted  are zoomed $71$ and $640$ thousand times in the top and bottom figures respectively. The leftmost column shows the original contact matrices.
  \cite{Rao2014} normalize the $n \times n$ contact matrix $C$ via the matrix balancing method of \citep{Knight2013} so that its columns and rows sum to $\sum_{i, j}{C_{i, j}}/n$. The resulting normalized matrices are depicted in the rightmost column.
  The middle column depicts the results of Algorithm \ref{alg:admm} for normalizing the matrices so that its columns and rows sum to $\sum_{i, j}{C_{i, j}}/n$. The Conjugate Gradient method is used to solve the linear system \eqref{eqn:reduced_linear_system} since using a Cholesky factorization resulted in memory issues. A tolerance of $10^{-3}$ is used for the termination of our Algorithm. The normalization of the $5$Kb and $1$Kb contact matrix take $701$ and $3136$ seconds respectively on a single-threaded implementation on Intel Gold 5120, 192GB memory.
  }
  \label{fig:hi-c}
\end{figure}

\subsection{Spectral clustering problems}
Next, we present results of running our Algorithm on correlation matrices arising from spectral clustering \citep{Zass2007}. In spectral clustering one is given a set of points $\{x_i \in \mathbb{R}^d\}$ to be arranged into $l$ clusters. To this end, an \emph{affinity matrix} $C$ is generated with each entry $C_{i, j}$ representing a measure of the pairwise similarity between points $i$ and $j$. This matrix is then normalized and used by later stages of the clustering procedure. \cite{Zass2007} suggested that the normalization
\begin{equation}\label{eqn:approximation_problem_zass}
  \begin{array}{ll}
    \mbox{minimize}   & \frac{1}{2}\norm{X - C}^2 \\
    \mbox{subject to} & X \; \mbox{nonnegative} \\
                      & X \mathbf{1}_n = \mathbf{1}_n,\,\,
                      X^T \mathbf{1}_n = \mathbf{1}_n,
 \end{array}
\end{equation}
 leads to superior clustering performance in various different test cases. Note that solving \eqref{eqn:approximation_problem_zass} is equivalent to solving \ref{eqn:approximation_problem} for $C + \epsilon_{\text{p}} \mathbf{1}_{n \times n}$. Note that the formulation \eqref{eqn:approximation_problem_zass} of \cite{Zass2007} does not exploit sparsity in $C$. Nevertheless, \cite{Zass2007} suggested that the following iterative scheme can be used to solve \ref{eqn:approximation_problem_zass}
\begin{subequations}\label{eqn:zaas}
\begin{align}
  \nonumber
  \tilde X^{k} &= X^k + n^{-2}(\mathbf{1}_{n}^T X^k\mathbf{1}_{n} + n)\mathbf{1}_{n \times n}\\
  &\quad\quad\;\;\;-n^{-1}(X^k \mathbf{1}_{n \times n} + \mathbf{1}_{n \times n}X^k)\\
  X^{k + 1} &= \Pi_+\left(\tilde X^{k}\right).
\end{align}
\end{subequations}
The first step in \eqref{eqn:zaas} minimizes the objective of \ref{eqn:approximation_problem} subject to the equality constraints, while the second projects the iterate to the nonnegative orthant. However, this approach is not guaranteed to solve \ref{eqn:approximation_problem} to optimality. For example, in the simple case of $C = \frac{1}{10}\tiny\begin{bmatrix} 1 & 9 & 9 \\ 9 & 1 & 0 \\ 9 & 0 & 9 \end{bmatrix},$ \eqref{eqn:zaas} converges to a suboptimal point $\bar X$ with $\norm{\bar X - X^*}_{\infty} = 0.0\overline{7}$ where $X^*$ is the optimizer. Nevertheless, it appears that, in general, \eqref{eqn:zaas} converges to a feasible point. However, even convergence to a feasible point can be much slower than our approach, as demonstrated in Figure \ref{fig:alternating-projections}.
  \begin{figure}
    \centering
    {\includegraphics[width=.40\textwidth]{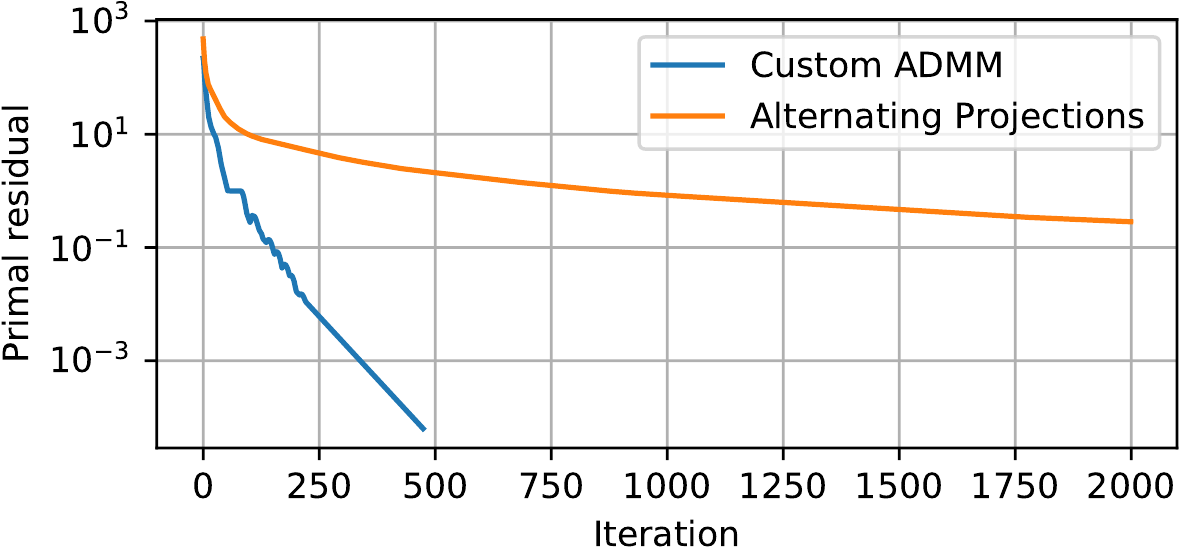}}
    \caption{Comparison of the primal convergence (i.e.\ feasibility) of Algorithm \ref{alg:admm} vs.\ the  approach of \citep{Zass2007} \eqref{eqn:zaas} on the Spambase dataset with \eqref{eqn:kernel} as affinity criterion with $\sigma = 100$.}
    \label{fig:alternating-projections}
    \vspace{-2ex}
  \end{figure}

  Besides guaranteed convergence to the optimizer of \ref{eqn:approximation_problem}, our approach can also handle sparsity in the affinity matrices. To demonstrate how exploiting sparsity can lead to significant speedups we consider the Spambase dataset\footnote{\texttt{archive.ics.uci.edu/ml/datasets/spambase}} (considered in \cite{Zass2007}) with an RBF kernel as an affinity criterion
  \begin{equation} \label{eqn:kernel}
    C_{i, j}(\sigma) = e^{-\norm{x_i - x_j}^2/\sigma^2},
  \end{equation}
  where $\sigma$ is a parameter that is typically tuned to achieve the best clustering performance. Note that due to the exponential form of $C(\sigma)$, some values will be very small. Therefore, we truncate to zero all entries with value less than $10^{-7}$. The runtimes of applying our Algorithm to this dataset are listed in Table \ref{tab:clustering} and compared to timings achieved with Gurobi (note that we use \ref{eqn:approximation_problem_reduced} for Gurobi as we consider \ref{eqn:approximation_problem_reduced} to be a ``standard'' reformulation of \ref{eqn:approximation_problem}). We observe that exploiting sparsity can lead to significant speedup as compared to treating the affinity matrix as fully dense, even if we follow the approach of \S\ref{subsubsec:generalizations}. At the same time, the optimizers of \ref{eqn:approximation_problem} for the affinity matrices $C(\sigma)$ considered in Table \ref{tab:clustering} appear to coincide with the ones for the fully dense $C(\sigma) + \epsilon_\text{p} \mathbf{1}_{n \times n}$.

\begin{table}
  \caption{Normalizing correlation matrices with Algorithm \ref{alg:admm} for spectral clustering on Spambase. A tolerance of $10^{-4}$ is used for the termination of our Algorithm. The Timings are expressed in seconds. and compared against Gurobi (on \ref{eqn:approximation_problem_reduced}) and against solving $C + \epsilon_{\text{p}} \mathbf{1}_{n \times n}$ with the approach of \S\ref{subsubsec:generalizations} where and is $C(\sigma)$ the original affinity matrix. Hardware used: Intel i7-5557U CPU @ 3.10GHz, 8GB Memory.}
  \label{tab:clustering}
  \sisetup{round-mode=places}
  \tabcolsep=0.06cm
  \begin{tabular}{
  l|
  *{5}{S[scientific-notation=true, round-precision=1, table-format=1.1e-2]}
  }
    \toprule
   $\sigma$ & 1.0 & {5.0} & {10.0} & {20.0} & \\
   $n_{\textrm{nz}}$ & 38822 & 1768350 & 4046782 & 7282984 \\
   \midrule
   $t_{\text{admm}}$ & 0.108655733 & 5.73347433 & 14.8491134 & 25.68104007 \\
   $t_{\text{gurobi}}$ & 0.393993741 & 46.33542767 & 109.3428531 & 195.512478
   \\
   $t_{\text{admm}}^{\text{dense}}$ &
   769.8787659 & 691.7524112 & 638.769045 & 718.8329326
  \end{tabular}
\end{table}

\subsection{Matrices from the SuiteSparse Collection}
Finally, we consider all Undirected Weighted Graph Matrices, with less than 50 million nonzeros, contained in the SuiteSparse collection\footnote{Available at \texttt{sparse.tamu.edu}}. $69$  matrices meet these criteria. We use Algorithm \ref{alg:admm} to normalize every matrix $C$ so that all of its columns and rows sum to $\max_{i, j}(C_{i, j})$. All of the problems, except two, have nonnegative entries. For these two exceptions, we change the negative entries to their absolute value, as we are unaware of practical cases where $C$ is expected to have negative entries.

Detailed comparison of our results with Gurobi presented in the Appendix. Figure \ref{fig:suitesparse} shows the timings achieved by our method, and the speedups relative to Gurobi (on \ref{eqn:approximation_problem_reduced}), as a function of each problem's nonzeros.
\begin{figure}[h!]
 	\centering
 	\includegraphics[width=.42\textwidth]{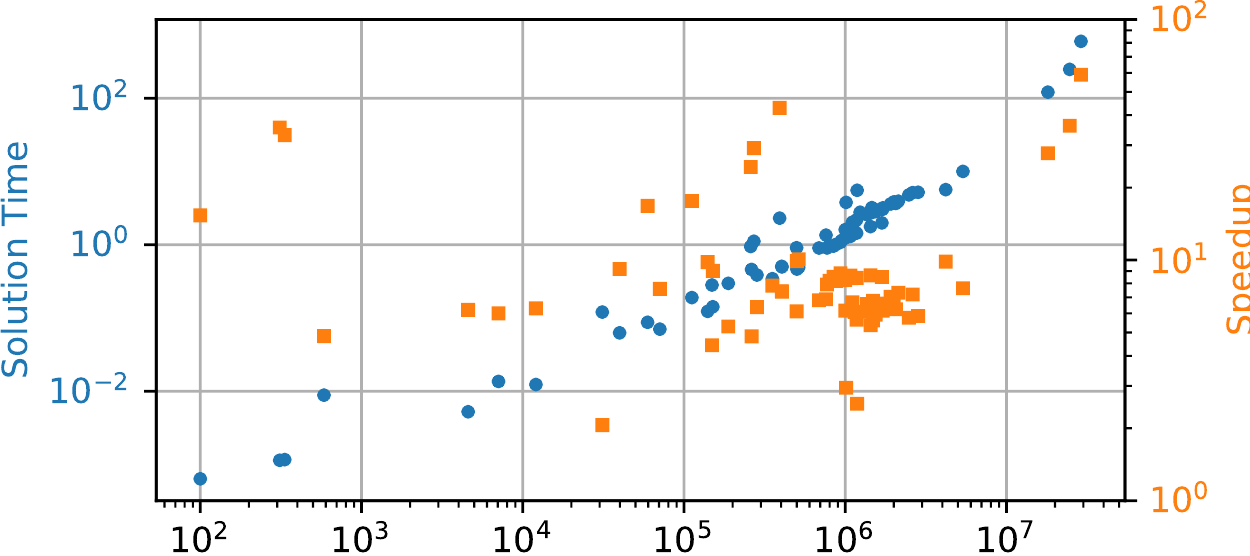}
 	\caption{Results of Algorithm \ref{alg:admm} for matrices from the SuiteSparse collection. Dots represent the timing of our Algorithm (with $10^{-4}$ tolerance), while squares represent the speedup achieved over Gurobi (on \ref{eqn:approximation_problem_reduced}). Hardware used: a single thread running on an Intel Gold 5120 with 192GB of memory.
  }
   \label{fig:suitesparse}
\end{figure}
\paragraph{Acknowledgements} This work was supported by the EPSRC AIMS CDT grant EP/L015987/1 and Schlumberger. Many thanks to Giovanni Fantuzzi for bringing this problem into our attention.

\bibliographystyle{plainnat}
\balance
\bibliography{../../Bibliography/bibliography}
\onecolumn
\appendix
\section{Proof of Theorem 3.2}\label{app:matrices}
\vspace{-2ex}
In this section we will prove Theorem 3.2 from the main paper, i.e.\ we will show that
\begin{equation*}
  A\diag(H_u \vc(D))A^T = S \odot (D + D^T) + \diag(S \odot (D + D^T) \mathbf{1})
\end{equation*}
where $S$ $(S_u)$ is an integer $(0-1)$ matrix representing the sparsity pattern of $C$ $(C_u)$ and $D$ is an upper triangular $n \times n$ matrix.

Using the fact that $H_u H_u^T = I$, the definition of $A$ and defining $d \eqdef H_u \vc D$ we get
\begin{align*}
  A \diag(d) A^T &= A \diag(\sqrt{d}) H_u H_u^T \diag(\sqrt{d}) A^T \\
  &= \left(A\diag(\sqrt{d})\diag(H_u)\right) \left(A\diag(\sqrt{d})\diag(H_u)\right)^T
\end{align*}
where $\sqrt{d}$ denotes the element-wise square root of $d$. Recalling that $A = A_1 + A_2$ we have
\begin{align*}
  A\diag(\sqrt{d})\diag(H_u) &= \underbrace{A_1\diag(\sqrt{d})\diag(H_u)}_{\eqdef B_1} + \underbrace{A_2\diag(\sqrt{d})\diag(H_u)}_{\eqdef B_2} \\
  \intertext{thus}
  A \diag(d) A^T
  &=
  \begin{bmatrix}
    I_n & I_n
  \end{bmatrix}
  \begin{bmatrix}
    B_1\\ B_2
  \end{bmatrix}
  \begin{bmatrix}
    B_1\\ B_2
  \end{bmatrix}^T
  \begin{bmatrix}
    I_n\\
    I_n
  \end{bmatrix}
\end{align*}
We now focus on the first part of the above symmetric product,
To this end, define $\sqrt{s}_i$ the $i^\textrm{th}$ column of $S_u \odot \sqrt D$ where $\sqrt D$ denotes the element-wise square root of $D$. Using the fact that $H_u^T\diag(x)H_u = \diag(\vc(S_u) \odot (H_u^T x))$ for any vector $x$ of appropriate dimensions we get
\begin{align}
  \nonumber
  H_u^T\diag(\sqrt d)H_u & = \diag(\vc(S_u) \odot (H_u^T \sqrt d))
  = \diag(\vc(S_u) \odot (H_u^T H \vc(\sqrt D))) \\
  \label{eqn:HtH}
  &= \diag(\vc(S_u \odot S_u \odot \sqrt D))
  = \diag(\vc(S_u \odot \sqrt D)).
\end{align}
Use the definitions of $A_1, A_2$ and \eqref{eqn:HtH} to get
\begin{align*}
  \begin{bmatrix}
    B_1 \\
    B_2
  \end{bmatrix}
  =
  \begin{bmatrix}
    \mathbf{1}^T \otimes I \\
    I \otimes \mathbf{1}^T
  \end{bmatrix}
  H_u^T\diag(d)H_u
  &=
  \begin{bmatrix}
    I & \cdots & I \\
    \mathbf{1}^T & & \\
    & \ddots & \\
    & & \mathbf{1}^T
  \end{bmatrix}
  \begin{bmatrix}
    \diag(\sqrt{s}_1) & & \\
    & \ddots & \\
    & & \diag(\sqrt{s}_n)
  \end{bmatrix} \\
  &=
  \begin{bmatrix}
    \diag(\sqrt{s}_1) & \cdots & \diag(\sqrt{s}_n) \\
    \sqrt{s}_1^T & & \\
    & \ddots & \\
    & & \sqrt{s}_n^T
  \end{bmatrix}
\end{align*}
We can then form the symmetric product of this matrix with itself to obtain
\begin{align*}
  \begin{bmatrix}
    B_1\\ B_2
  \end{bmatrix}
  \begin{bmatrix}
    B_1\\ B_2
  \end{bmatrix}^T &=
  \begin{bmatrix}
    \diag(\sqrt{s}_1) & \cdots & \diag(\sqrt{s}_n) \\
    \sqrt{s}_1^T & & \\
    & \ddots & \\
    & & \sqrt{s}_n^T
  \end{bmatrix}
  \begin{bmatrix}
    \diag(\sqrt{s}_1) & \sqrt{s}_1 & & \\
    \vdots & & \ddots & \\
    \diag(\sqrt{s}_n) & & & \sqrt{s}_n\\
  \end{bmatrix} \\
  \intertext{or, noting that $\sqrt{s}_i^T \sqrt{s}_i = \mathbf{1}^T s_i$, $\diag^2(\sqrt{s}_i) = \diag(s_i)$, $\diag(\sqrt{s}_i)\sqrt{s}_i = s_i$ where $s_i$ denotes the $i-$th column of $S_u \odot \mat(d)$,}
  \begin{bmatrix}
    B_1\\ B_2
  \end{bmatrix}
  \begin{bmatrix}
    B_1\\ B_2
  \end{bmatrix}^T &=
  \begin{bmatrix}
    \diag\left(\sum_i s_i\right) & s_1 & \dots & s_n \\
    s_1^T & \mathbf{1}^T s_1 & & \\
    \vdots & & \ddots & \\
    s_n^T & & & \mathbf{1}^T s_n
  \end{bmatrix}
  =
  \begin{bmatrix}
    \diag((S_u \odot D) \mathbf{1}) & (S_u \odot D) \\
    (S_u \odot D) ^T & \diag((S_u \odot D)^T\mathbf{1})
  \end{bmatrix}.
\end{align*}
Thus
\begin{align*}
  A\diag(d)A^T &=
  \begin{bmatrix}
    I & I
  \end{bmatrix}
  \begin{bmatrix}
    B_1\\ B_2
  \end{bmatrix}
  \begin{bmatrix}
    B_1^T B_2^T
  \end{bmatrix}
  \begin{bmatrix}
    I\\
    I
  \end{bmatrix} \\
  &= \diag((S_u \odot D) \mathbf{1}) + (S_u \odot D) +
  (S_u \odot D) ^T + \diag((S_u \odot D)^T\mathbf{1})
  \end{align*}
or, recalling that $D$ is upper triangular,
$$
  A\diag(H_u \vc(D))A^T =
  S \odot (D + D^T) + \diag(S \odot (D + D^T).
  $$

\section{Detailed Results for the SuiteSparse Matrices}
In this Section we provide detailed results for \S 4.3 of the main paper.
\newcommand{\tableheader}{{Problem Name} & {$n$} & {$n_{\textrm{nz}}$} & {$t_{\text{admm}}$} & {$t_{\text{gurobi}}$} \\}
{
  {\sisetup{round-mode=places, scientific-notation=true}
  {
	\begin{longtable}{
		l
  *{2}{S[scientific-notation=false, round-precision=0, table-number-alignment=center, table-column-width=1.5cm]}
  *{2}{S[round-precision=2, table-format=1.2e-2]}
  }
  \caption{Results of Algorithm 1 for Undirected Weighted Graph Matrices from the SuiteSparse collection. A tolerance of $10^{-4}$ is used for the termination of our Algorithm. The timings are in seconds and they are compared against the solution times of Gurobi (on $\mathcal{P}1$). Hardware used: a single thread running on an Intel Gold 5120 with 192GB of memory.}\label{tab:sdplib}\\
  \toprule \tableheader \midrule \endfirsthead
  \caption*{Table \ref{tab:sdplib}: Continued.}\\
	\toprule \tableheader \midrule \endhead
  \begin{filecontents*}{./data/suitesparse.csv}
name,n,nnz,time,timeGurobi,iter
GD97\_b,47,311,0.001144583,0.040713389,25
Journals,124,12068,0.012380451,0.078004306,50
MISKnowledgeMap,2427,59449,0.087474743,1.468455802,25
Sandi\_authors,86,334,0.001168916,0.038663428,25
Stranke94,10,100,0.000639731,0.009813806,25
USAir97,332,4584,0.005261202,0.032661117,50
ak2010,45292,262390,0.460235099,2.21760323,50
al2010,252266,1482748,2.872556729,16.13918514,25
ar2010,186211,1090521,1.391302397,11.68573415,25
astro-ph,16706,259208,0.946789527,23.06052008,75
az2010,241666,1437760,2.712681215,14.54257457,50
ca2010,710145,4199511,5.652389059,55.65560056,25
co2010,201062,1175636,2.186754113,12.36615737,50
cond-mat-2003,31163,271221,1.116504868,32.57890618,50
cond-mat-2005,40421,391803,2.307912032,99.04085359,50
cond-mat,16726,111914,0.190213398,3.346612252,25
ct2010,67578,403930,0.507057159,3.761941705,25
de2010,24115,140171,0.123380632,1.209063249,25
fl2010,484481,2830775,5.198781793,30.42785321,50
ga2010,291086,1709142,3.183507048,21.01848481,25
geom,7343,31139,0.120673857,0.248778453,50
hep-th,8361,39863,0.062775013,0.576367815,25
hi2010,25016,149142,0.282136183,1.249069235,25
human\_gene1,22283,24669643,247.2623819,8944.381835,275
human\_gene2,14340,18068388,120.9833797,3360.773358,275
ia2010,216007,1237177,2.779744777,16.94632424,25
id2010,149842,878106,1.013851658,8.304057426,25
il2010,451554,2616018,5.138113174,36.97040179,25
in2010,267071,1548787,3.001716927,19.15351905,25
ks2010,238600,1360398,2.576300823,16.92378856,25
ky2010,161672,949450,1.143762127,9.4367977,25
la2010,204447,1185081,5.5368203,13.99973217,150
lesmis,77,585,0.008865166,0.042803149,25
ma2010,157508,934118,1.078702949,9.494873778,25
md2010,145247,845625,0.959322742,8.165351944,25
me2010,69518,404994,0.498718117,3.685296804,25
mi2010,329885,1907975,3.595207749,25.26441313,25
mn2010,259777,1486879,2.889366046,19.54855028,25
mo2010,343565,2000133,3.841074355,25.65579977,25
mouse\_gene,45101,28967291,597.4126569,35179.61984,250
ms2010,171778,1011758,3.78495657,11.15516146,50
mt2010,132288,770956,0.90461116,7.160880951,25
nc2010,288987,1705607,3.046944489,18.7783243,25
nd2010,133769,759715,1.353218078,9.299905982,25
ne2010,193352,1107206,2.031238861,13.54695994,25
netscience,1589,7073,0.013645757,0.081943336,25
nh2010,48837,283387,0.386556953,2.463082339,25
nj2010,169588,999500,1.234814493,10.20678733,25
nm2010,168609,999579,1.607664475,9.915786193,50
nopoly,10774,70842,0.070538662,0.535279142,25
nv2010,84538,501536,0.465651838,4.612488855,25
ny2010,350169,2059713,3.664001605,22.88572352,25
oh2010,365344,2133584,3.931611946,28.76144384,25
ok2010,269118,1543266,2.827499783,16.75753399,25
or2010,196621,1176133,1.443359479,12.17525496,25
pa2010,421545,2480007,4.796455776,27.62559279,25
rgg\_n\_2\_15\_s0,32768,353248,0.3442263,2.694639872,25
ri2010,25181,150931,0.14245204,1.285132196,25
sc2010,181908,1075068,1.299683664,11.19900678,25
sd2010,88360,499082,0.910191883,5.572339029,25
tn2010,240116,1434082,1.77045163,15.28289146,25
tx2010,914231,5370503,10.01990172,76.48504215,25
ut2010,115406,687472,0.903559452,6.149731095,25
va2010,285762,1687890,1.987615492,16.89796242,25
vt2010,32580,188178,0.297587111,1.576575066,25
wa2010,195574,1143006,2.053336963,12.51075232,50
wi2010,253096,1462500,3.209248976,19.75003686,25
wv2010,135218,798140,0.974628354,7.988268728,25
wy2010,86204,513790,0.489159806,4.927791267,25
\end{filecontents*}
\csvreader[head to column names]{./data/suitesparse.csv}{}
  {\\
  \name	& \n & \nnz	& \time	& \timeGurobi
  }
  \\\hline
  \end{longtable}}
  }
}
\vfill

\end{document}